\documentclass{article}
\usepackage[utf8]{inputenc}
\usepackage{amsmath, amssymb, enumerate, multicol, chngcntr}
\usepackage[exercise]{amsthm}
\usepackage{xcolor}
\usepackage{pdfpages}
\usepackage{comment}
\usepackage{caption}
\usepackage{bm}
\usepackage{mathtools}

\theoremstyle{plain}
\newtheorem{thm}{Theorem}
\newtheorem{cor}[thm]{Corollary}
\newtheorem{prop}[thm]{Proposition}
\newtheorem{lemma}[thm]{Lemma}

\counterwithin{equation}{section}

\usepackage{fancyhdr}
\usepackage{setspace}

\usepackage[english]{babel}
\usepackage[euler]{textgreek}
\usepackage{stmaryrd}

\usepackage{tikz}
\usetikzlibrary{shapes,arrows}
\usepackage{tikz-cd}
\usepackage{amssymb}
\usepackage{graphicx}
\graphicspath{{Images/}}
\usepackage[font=small,labelfont=bf]{caption}
\usepackage{pgf,tikz,pgfplots}
\pgfplotsset{compat=1.15}
\usepackage{mathrsfs}
\usepackage{diagbox}
\usepackage{float}
\usepackage{xcolor}
\usepackage{hyperref}
\usepackage{ragged2e}
\usepackage{enumerate}
\usepackage{upgreek}
\usepackage{multicol}

\theoremstyle{remark}

\theoremstyle{plain}

\newtheoremstyle{note}
  {3pt}
  {3pt}
  {}
  {}
  {\itshape}
  {:}
  {.5em}
  {}

\newtheoremstyle{citing}
  {3pt}
  {3pt}
  {\itshape}
  {}
  {\bfseries}
  {.}
  {.5em}
  {\thmnote{#3}}

\theoremstyle{citing}

\newtheoremstyle{break}
  {9pt}
  {9pt}
  {\itshape}
  {}
  {\bfseries}
  {.}
  {\newline}
  {}

\let\lvert=|\let\rvert=|

\title{Infinite groups with isomorphic power graph and commuting graph}

\author{Surbhi \footnote{Dr. B. R. Ambedkar University Delhi, Delhi 110006; \ E-mails: surbhi.21@stu.aud.ac.in, surbhi.ts19@gmail.com} \ and \ Geetha Venkataraman\footnote{Corresponding author, Dr. B. R. Ambedkar University Delhi, Delhi 110006; E-mails: geetha@aud.ac.in, geevenkat@gmail.com}}

\date{}
\begin{document}
\fontfamily{cmr}\selectfont
\maketitle

\bigskip
\noindent
{\small{\bf ABSTRACT:}}
 In this paper, we investigate certain graphs defined on groups, with a focus on infinite groups. The graphs discussed are the power graph, the enhanced power graph, and the commuting graph whose vertex set is a group $G$. The power graph is a graph in which two vertices are adjacent if one is some power of the other. In the enhanced power graph, an edge joins two vertices if they generate a cyclic subgroup of $G$. In the commuting graph, two vertices are adjacent if they commute in $G$. We prove a necessary and sufficient condition for any two of these graphs to be equal. This extends existing results for finite groups. In addition, we show that the power graph of the locally quaternion group is isomorphic to the commuting graph of the locally dihedral group. Lastly, we also answer a question posed by P. J. Cameron about the existence of groups $G_1$ and $G_2$ both of whom have power graph not equal to commuting graph but the power graph of $G_1$ and the commuting graph of $G_2$ are isomorphic.


\medskip
\noindent
{\small{\bf Mathematics Subject Classification-MSC2020}{:} }
20D60, 05C25, 20K10, 20F50, 20K21

\baselineskip=\normalbaselineskip


\section{Introduction}

Throughout the paper, by a graph, we mean an undirected simple graph. The power graph of a group $G$ is defined as a graph with vertex set $G$ and two vertices are adjacent if one is some power of the other. It was first defined by Kelarev and Quinn for semigroups \cite{Kelarev}. We denote it as Pow($G$). In \cite{Abdollahi}, Abdollahi and Hassanabadi introduced a graph on a group $G$ in which two vertices were joined by an edge if they generated a non-cyclic subgroup of the group $G$. They removed the isolated vertices from the vertex set of this graph. They called this the noncyclic graph of a group $G$. Later, in 2017, the complement of the noncyclic graph was independently defined in \cite{Aalipour} and was termed as the enhanced power graph. The enhanced power graph of a group $G$ is a graph with $G$ as its vertex set, and an edge joins two vertices if they generate a cyclic subgroup. It is denoted as EPow($G$). The isolated vertices are not removed from the vertex set of this graph. In 2021, P. J. Cameron \cite{PJC_reviewpaper} termed the sequence: power graph, enhanced power graph and commuting graph as the hierarchy since there are inclusions between these graphs. Let $E(\Gamma)$ denote the edge set of a graph $\Gamma$. For graphs $\Gamma_1$ and $\Gamma_2$, the notation $E(\Gamma_1) \subseteq E(\Gamma_2)$ implies that vertex set of $\Gamma_1$ is the same as vertex set of $\Gamma_2$ and $\Gamma_1$ has some of the edges of $\Gamma_2$, that is, $\Gamma_1$ is a spanning subgraph of $\Gamma_2$. If $x$ and $y$ are adjacent in Pow($G$), then clearly, the subgroup $\langle x,y \rangle$ is cyclic. Therefore, the vertices $x$ and $y$ are adjacent in EPow($G$) and we have $E($Pow($G$)) $\subseteq$ $E$(EPow($G$)). This inclusion can be strict as can be seen in the case of $\mathbb{Z}_6$, the cyclic group of order 6. Aalipour et al. \cite{Aalipour} proved that the power graph and the enhanced power graph of a finite group $G$ are equal if and only if $G$ does not have a subgroup isomorphic to $C_p \times C_q$ for distinct primes $p$ and $q$. We extend the result for any group $G$.

\begin{thm}\label{thm1}
    The power graph and the enhanced power graph of a group $G$ are equal if and only if $G$ does not have a subgroup isomorphic to $\mathbb{Z}$ or $C_p \times C_q$ where $p,q$ are distinct primes.
\end{thm}

The commuting graph of a group $G$ is a graph with vertex set $G$, and two vertices $x$ and $y$ are adjacent if $xy=yx$. We denote this graph as Com($G$). This graph was introduced in 1955 by Brauer and Fowler \cite{Brauer}. For a group $G$, if two vertices $x,y$ are adjacent in EPow($G$) then they generate a cyclic subgroup of $G$. So, the elements $x$ and $y$ commute with each other. Therefore, they are adjacent in Com($G$). We get that $E$(EPow($G$)) $\subseteq$ $E$(Com($G$)). In \cite{Aalipour}. Aalipour et al. gave a necessary and sufficient condition for a finite group $G$ to satisfy EPow($G$) = Com($G$). The enhanced power graph and the commuting graph of a finite group $G$ are equal if and only if the group $G$ does not have a subgroup isomorphic to $C_p \times C_p$ for any prime $p$. We extend this to any group $G$.

\begin{thm}\label{thm2}
    The enhanced power graph and the commuting graph of a group $G$ are equal if and only if $G$ does not have a subgroup isomorphic to $\mathbb{Z} \times \mathbb{Z}$, $\mathbb{Z} \times C_p$ or $C_p \times C_p$ where $p$ is prime.
\end{thm}

Combining the above two theorems, we get the following result for the equality of power graph and commuting graph of $G$.

\begin{thm}\label{thm3}
    The power graph and the commuting graph of a group $G$ are equal if and only if $G$ does not have a subgroup isomorphic to $\mathbb{Z}$, $\mathbb{Z} \times \mathbb{Z}$, $\mathbb{Z} \times C_p$ or $C_p \times C_q$ where $p,q$ are primes, not necessarily distinct.
\end{thm}

The above results give us necessary and sufficient conditions for any group $G$ to have two of the three graphs equal. But what can we say about two different graphs of two different groups being equal? In \cite{PJC_reviewpaper}, the author P. J. Cameron asked the following question: \textit{Do there exist non-isomorphic groups $G$ and $H$ such that} Pow($G$)\textit{ is isomorphic to} Com($H$)\textit{?} We answered this question in \cite{Paper_1} with the generalised quaternion group and the dihedral group. We also observed that all three graphs of the generalised quaternion group are equal but for the dihedral group of 2-power order, the enhanced power graph equals the power graph but not the commuting graph. We proved that the power graph of the generalised quaternion group is isomorphic to the commuting graph of the dihedral group of the same order. In this paper, we consider the infinite counterparts of these groups and observe the equality amongst the graphs defined on them. The union of generalised quaternion groups of every order forms an infinite group called the locally quaternion group which we will denote by $Q_{2^{\infty}}$. Similarly, the union of all dihedral groups of order $2^n$ forms the locally dihedral group and is denoted by $D_{2^{\infty}}$. In this paper, we answer P. J. Cameron's question (stated above) for these infinite groups.

\begin{thm}\label{thm4}
    Let $Q_{2^{\infty}}$ be the locally quaternion group and $D_{2^{\infty}}$ be the locally dihedral group then {\rm Pow(}$Q_{2^{\infty}})$ is isomorphic to {\rm Com(}$D_{2^{\infty}})$.
\end{thm}

Next, in \cite{Paper_1}, we considered the dicyclic groups because the generalised quaternion groups are a special case of the dicyclic group. The goal was to see whether the power graph of the dicyclic group is isomorphic to the commuting graph of the dihedral group of the same order. We observed that for the dicyclic group, its enhanced power graph is not equal to its power graph but is equal to its commuting graph. However, the three graphs are unequal for the dihedral group of order $4m$ where $m > 2$. In this paper, we consider the infinite quaternion group, $Q_{\infty}$ which contains the dicyclic groups of every order. We try to draw similarities between the graphs defined on dicyclic group and the infinite quaternion group.

P. J. Cameron also asked a stronger version of the above-mentioned question, namely, \textit{Can Pow($G_1$) and Com($G_2$) be isomorphic for groups $G_1$ and $G_2$ both of whom have their power graph not equal to their commuting graph?} In this paper, we answer this question positively by giving two non-isomorphic infinite groups that meet all the conditions in the question posed.

\begin{thm}\label{thm5}
    Let $D_{\infty}$ be the infinite dihedral group and $D_{2^{\infty}}$ be the locally dihedral group. Then {\rm Pow(}$D_{2^{\infty}})$ is isomorphic to {\rm Com(}$D_{\infty})$. Furthermore, both $D_{\infty}$ and $D_{2^{\infty}}$ have respective power graphs not equal to their commuting graphs.
\end{thm}

In Section 2, we give the proofs of Theorems \ref{thm1} and \ref{thm2}. In Section 3, we describe the infinite groups $Q_{2^{\infty}}$, $D_{2^{\infty}}$, $Q_{\infty}$, $D_{\infty}$ and show the equalities within their graph hierarchy. Lastly, we prove Theorems \ref{thm4} and \ref{thm5}.

\section{Equality of Power, Enhanced Power and Commuting Graphs of a Group}

Aalipour et al. \cite{Aalipour} proved that for a finite group $G$, we have Pow($G$) = EPow($G$) if and only if $G$ does not contain a subgroup isomorphic to $C_p \times C_q$ for distinct primes $p$ and $q$. This result also holds for torsion groups as seen below. 

\setcounter{thm}{0}
\begin{thm}\label{theo1}
    The power graph and the enhanced power graph of a group $G$ are equal if and only if $G$ does not have a subgroup isomorphic to $\mathbb{Z}$ or $C_p \times C_q$ where $p,q$ are distinct primes.
\end{thm}

\begin{proof}
    Let $G$ be a torsion group, that is, every element of $G$ is of finite order. Let $x,y$ be commuting elements of $G$ such that $|x|=p$ and $|y|=q$ where $p$ and $q$ are distinct primes. Since $(|x|,|y|)=1$ and $xy=yx$ therefore $xy$ has order $pq$. So, we have $\langle x,y \rangle = \langle xy \rangle$. Therefore, $x$ and $y$ are adjacent in EPow($G$) but not in Pow($G$). Conversely, let $G$ not have a subgroup isomorphic to $C_p \times C_q$. Let $x$ be adjacent to $y$ in EPow($G$), then $H = \langle x,y \rangle$ is cyclic. Since $G$ does not have a cyclic subgroup whose order is divisible by two distinct primes, the order of $H$ is a prime-power. Since the subgroup lattice of a prime-power order cyclic group is a chain, we get either $\langle x \rangle \subseteq \langle y \rangle$ or $\langle y \rangle \subseteq \langle x \rangle$. Therefore, $x$ is adjacent to $y$ in Pow($G$). Hence, we have Pow($G$) = EPow($G$). 
    
    If $G$ is not a torsion group, then $G$ has a subgroup isomorphic to $\mathbb{Z}$, say $H$. Let $x \in G$ be the generator of $H$. Then, the elements $x^2$ and $x^3$ are not adjacent in Pow($G$) but are adjacent in EPow($G$) as any subgroup of a cyclic group is cyclic. Conversely, let Pow($G$) $\neq$ EPow($G$) for an infinite group $G$. Let $x,y \in G$ such that $x$ is joined to $y$ in EPow($G$) but not in Pow($G$). The subgroup $K = \langle x,y \rangle $ is cyclic but neither $x$ is a power of $y$ nor $y$ is a power of $x$. If $K$ is infinite, then $K \cong \mathbb{Z}$. If $K$ is of finite order, then $K \cong \mathbb{Z}_n$ for some natural number $n$. If $n$ a is prime-power then $x$ and $y$ are adjacent in Pow($G$) which is a contradiction. Therefore, $n$ is divisible by at least two distinct primes $p,q$. Hence, the group $G$ has a subgroup isomorphic to $C_p \times C_q$.
\end{proof}

Another result from \cite{Aalipour} states that a finite group $G$ satisfies EPow($G$) = Com($G$) if and only if $G$ does not have a subgroup isomorphic to $C_p \times C_p$ for any prime $p$. 

\begin{thm}\label{theo2}
    The enhanced power graph and the commuting graph of a group $G$ are equal if and only if $G$ does not have a subgroup isomorphic to $\mathbb{Z} \times \mathbb{Z}$, $\mathbb{Z} \times C_p$  or $C_p \times C_p$ where $p$ is prime.
\end{thm}

\begin{proof}
    Let $G$ be an infinite group and suppose that EPow($G$) $\subsetneq$ Com($G$). Then, there exists $x$ and $y$ in $G$ such that $A = \langle x,y \rangle$ is not cyclic but $xy=yx$. Since $A$ is a finitely generated abelian group which is not cyclic, we will get that $A \cong \mathbb{Z} \times \mathbb{Z}$ or $A \cong \mathbb{Z} \times \mathbb{Z}_n$ for some natural number $n$ if $A$ is infinite. If $A$ is finite, then $A$ will have a subgroup isomorphic to $C_p \times C_p$ for some prime $p$. Thus, the group $G$ has a subgroup isomorphic to $\mathbb{Z} \times \mathbb{Z}$ or $\mathbb{Z} \times C_p$  or $C_p \times C_p$ for some prime $p$.
    Conversely, for $A = \mathbb{Z} \times \mathbb{Z} \mbox{ or } \mathbb{Z} \times C_p  \mbox{ or } C_p \times C_p $, the enhanced power graph of $A$ is not equal to its commuting graph. So, we will have EPow($G$) $\neq$ Com($G$) whenever $G$ has a subgroup isomorphic to $A$.
\end{proof}

\begin{cor}\label{hereditary}
    Suppose $G_1 \subseteq G_2 \subseteq \cdots \subseteq G_k \subseteq \cdots$ is a chain of groups. Let $G = \bigcup_{n \in \mathbb{N}} G_n $. Then, we have 
    \begin{enumerate}[{\rm (i)}]
        \item {\rm Pow(}$G${\rm ) = EPow(}$G${\rm )} if and only if {\rm Pow(}$G_i${\rm ) = EPow(}$G_i${\rm )} for every $i \in \mathbb{N}$. 
        \item {\rm EPow(}$G${\rm ) = Com(}$G${\rm )} if and only if {\rm EPow(}$G_i${\rm ) = Com(}$G_i${\rm )} for every $i \in \mathbb{N}$. 
    \end{enumerate}
\end{cor}
\begin{proof}
    Let Pow($G_i$) = EPow($G_i$) for every $i \in \mathbb{N}$, that is for every $i$, the subgroup  $G_i$ does not contain a subgroup isomorphic to $\mathbb{Z}$ or $C_p \times C_q$ for distinct primes $p$ and $q$. Let if possible Pow($G$) $\neq$ EPow($G$). Then $G$ has a subgroup isomorphic to $\mathbb{Z}$ or $C_p \times C_q$ for some distinct primes $p$ and $q$, say $\langle y \rangle$. The element $y$ is either of infinite order or of order $pq$ respectively. The element $y$ belongs to $G_i$ for some $i$ which contradicts our assumption. For the second part, if EPow($G_i$) = Com($G_i$) for every $i \in \mathbb{N}$ and EPow($G$) $\neq$ Com($G$), then $G$ has a subgroup isomorphic to $\mathbb{Z} \times \mathbb{Z}$, $\mathbb{Z} \times C_p$  or $C_p \times C_p$ for some prime $p$, say $H$, whereas none of the $G_i$'s do. If $H$ is isomorphic to $\mathbb{Z} \times \mathbb{Z}$, $\mathbb{Z} \times C_p$, then $H$ has an element of infinite order which belongs to $G_i$ for some $i$ and that contradicts our assumption. If $H$ is isomorphic to $C_p \times C_p$, then $H$ has two distinct commuting elements of order $p$ which do not generate a cyclic group. This is also a contradiction because none of the $G_i$'s have any such elements by assumption.
\end{proof}

\section{Infinite Groups}

The finite groups studied in \cite{Paper_1} were the generalised quaternion group $Q_{2^{n+1}}$, the dicyclic group $Q_{4m}$ and the dihedral group $D_{2.2m}$. We described the power graph of $Q_{2^{n+1}}$, the power graph and enhanced power graph of $Q_{4m}$ and the commuting graph of $D_{2.2m}$. 


The generalised quaternion group is defined as $ Q_{2^{n+1}} = \langle h,x \mid h^{2^n} = x^4 = e , h^{2^{n-1}} = x^2 , x^{-1}hx = h^{-1} \rangle .$ A finite $p$-group having a unique subgroup of order $p$ is either cyclic or a generalised quaternion group. Since $Q_{2^{n+1}}$ has a unique element of order 2, namely $x^2$, its non-trivial non-cyclic subgroups are also generalised quaternion groups of smaller order. So, we get the following ascending series: $ Q_8 \subseteq Q_{16} \subseteq Q_{32} \subseteq Q_{64} \subseteq \cdots \subseteq \bigcup_{n=3}^{\infty} Q_{2^n} $. If we define $Q_2 = Q_4 = Q_8$, then $\bigcup_{n=3}^{\infty} Q_{2^n} = \bigcup_{n \in \mathbb{N}} Q_{2^n}$. We denote this union as $Q_{2^{\infty}}$, and we call this the locally quaternion group.  Note that this group is a torsion group, as the order of every element is some power of 2. Let $C_{2^{\infty}}$ denote the Pr\"{u}fer 2-group, an infinite abelian locally cyclic 2-group which is the union of a chain of cyclic 2-groups of order $2, 2^2, 2^3 , \ldots$. The group $C_{2^{\infty}}$ can be presented as $C_{2^{\infty}} = \langle g_1,g_2,g_3, \ldots \mid g_1^2=1, g_2^2=g_1,g_3^2=g_2, \ldots \rangle$. Note that each generalised quaternion group of order $2^{n+1}$ has a cyclic subgroup of order $2^n$. Hence, we get that $C_{2^{\infty}}$ is isomorphic to a subgroup of $Q_{2^{\infty}}$. Therefore, the locally quaternion group has the following presentation \cite{Kegel_Bertram}

\begin{equation*}
\begin{split}
Q_{2^{\infty}} = \langle C,x \mid  & \ C \cong C_{2^{\infty}} , x^{-1}cx= c^{-1} \mbox{ for every } c \in C, \\
 & \ x^2 = z \mbox{ where } z \mbox{ is the unique element of } C_{2^{\infty}} \mbox{ of order } 2  \rangle .
\end{split}
\end{equation*}



\noindent The order of every element of $Q_{2^{\infty}}$ is some power of 2 and the only finite subgroups of $Q_{2^{\infty}}$ are $C_{2^n}$ and $Q_{2^{n+1}}$ for all $n \in \mathbb{N}$. This is because if $A$ is a finite subgroup of $Q_{2^{\infty}}$, then $A$ is a subgroup of $Q_{2^{m+1}}$ for some natural number $m$, and since every subgroup of $Q_{2^{m+1}}$ is either cyclic of 2-power order or a generalised quaternion, therefore $A$ has to be either cyclic or generalised quaternion. Since the group $Q_{2^{n+1}}$ does not have a subgroup isomorphic to $C_p \times C_q$ for any primes $p \mbox{ and }q$, we have Pow($Q_{2^{n+1}}$) = EPow($Q_{2^{n+1}}$) = Com($Q_{2^{n+1}}$). Therefore, by Theorems \ref{theo1} and \ref{theo2}, we have $ \mbox{Pow}( Q_{2^{\infty}}) = \mbox{EPow}( Q_{2^{\infty}}) = \mbox{Com}( Q_{2^{\infty}})$. This can also be verified by using the Corollary \ref{hereditary}.\\

The dihedral group of order $2^{n+1}$ has the following presentation: $ D_{2.2^n} = \langle r,s \mid r^{2^n} = s^2 = e , srs = r^{-1} \rangle .$ It is well-known that $ D_8 \subseteq D_{16} \subseteq D_{32} \subseteq D_{64} \subseteq \cdots \subseteq \bigcup_{n \in \mathbb{N}} D_{2^n} = D_{2^{\infty}} $ where $D_2=D_4=D_8$. The group $D_{2^{\infty}}$ is called the locally dihedral group. It has the following presentation \cite{Kegel_Bertram}
$$ D_{2^{\infty}} = \langle C,s \mid C \cong C_{2^{\infty}}, s^2 = 1 , scs= c^{-1} \mbox{ for every } c \in C \rangle .$$

\noindent Since $D_{2^{\infty}}$ is a torsion group and the order of each element is a power of 2, it can not have subgroups isomorphic to $\mathbb{Z}$, $\mathbb{Z} \times \mathbb{Z}$, $\mathbb{Z} \times C_p$  or $C_p \times C_q$ for distinct primes $p$ and $q$. Since the 2-group $D_{2.2^n}$ contains $2^{n-2}$ subgroups isomorphic to $C_2 \times C_2$, we have $ \mbox{Pow}( D_{2^{\infty}}) = \mbox{EPow}( D_{2^{\infty}}) \subsetneq \mbox{Com}( D_{2^{\infty}}) $. This also follows from Corollary \ref{hereditary}.

This gives us the following result.
\begin{cor}\label{Cor_1}
    \begin{enumerate}[{\rm (i)}]
        \item The power graph, enhanced power graph and commuting graph of the locally quaternion group $Q_{2^{\infty}}$ are equal.
        \item The enhanced power graph of the locally dihedral group $D_{2^{\infty}}$ equals its power graph but not its commuting graph.
    \end{enumerate}
\end{cor}

We use the symbol $ \Gamma_1 \cup \Gamma_2$ to describe the disjoint union of graphs, the symbol $ \Gamma_1 \underline\cup \Gamma_2$ to describe the non-disjoint union of graphs and the symbol $ \Gamma_1 \nabla \Gamma_2$ to describe the graph join of graphs $\Gamma_1$ and $\Gamma_2$. The union of graphs $ \Gamma_1$ and $ \Gamma_2$ is a graph whose vertex set is the union of $V(\Gamma_1)$ and $V(\Gamma_2)$ and the edge set is the union of $E(\Gamma_1)$ and $E(\Gamma_2)$. In $ \Gamma_1 \cup \Gamma_2$, we consider the vertices and edges of $ \Gamma_1$ and $\Gamma_2$ distinct regardless of their labels, whereas in $ \Gamma_1 \underline\cup \Gamma_2$, we take the graph union by merging labelled vertices and edges. We consider $ \Gamma_1$ and $\Gamma_2$ with disjoint vertices to define the graph join. The graph join $ \Gamma_1 \nabla \Gamma_2$ is a graph whose vertex set is $V(\Gamma_1) \cup V(\Gamma_2)$ and the edge set is $E(\Gamma_1) \cup E(\Gamma_2)$ together with all the edges joining $V(\Gamma_1)$ and $V(\Gamma_2)$. The symbol $K_n$ is used for a complete graph on $n$ vertices.

We described the following structure of the power graph of $Q_{2^{n+1}}$ in \cite{Paper_1}. In Pow($Q_{2^{n+1}}$), the elements $e$ and $x^2$ are universal vertices. If we look at the induced subgraph on $Q_{2^{n+1}} \symbol{92} \{ e,x^2 \}$, we have one copy of $K_{2^n-2}$ and $2^{n-1}$ copies of $K_2$. All these complete graphs are disjoint from each other. We have Pow($Q_{2^{n+1}}$) = $(K_{2^n-2}$ $\cup$ $(2^{n-1})K_2)$ $\nabla$ $K_2$. Now, in Pow($Q_{2^{\infty}}$), we observe that the identity element and the unique involution are universal vertices, and after removing them, we get one copy of a complete graph on infinite vertices and infinite copies of $K_2$. Again, all these complete graphs are disjoint from each other. This gives us Pow($Q_{2^{\infty}}$) = $(K_{\infty}$ $\cup$ $(\infty)K_2)$ $\nabla$ $K_2$. Here $K_{\infty}$ represents a complete graph on countably infinitely many vertices. Similarly, $({\infty})K_2$ also represents countably infinitely many copies of $K_2$, a complete graph on 2 vertices. Figure \ref{com_g} provides a depiction of this. We observe that the structure of the power graph of $Q_{2^{n+1}}$ is carried forth to the power graph of its infinite counterpart, $Q_{2^{\infty}}$. By Theorem 1 of \cite{Paper_1}, the power graph of the generalised quaternion group is isomorphic to the commuting graph of the dihedral group of the same order. In Com($D_{2^{\infty}}$), the identity element and involution from the Pr\"{u}fer group are universal vertices. After removing them, the remaining vertices of the Pr\"{u}fer group form a complete graph and the remaining elements of $D_{2^{\infty}}$, that are of the form $xs$ where $x \in C$, form infinite copies of $K_2$. These complete graphs are again disjoint from each other. This implies that the structure of the commuting graph of $D_{2.2^n}$ is carried forth to the commuting graph of its infinite counterpart, $D_{2^{\infty}}$. 
We can now show that there exist non-isomorphic groups $G$ and $H$ of infinite order such that the power graph of $G$ is isomorphic to the commuting graph of $H$.

\setcounter{thm}{3}

\begin{thm}\label{theo4}
    Let $Q_{2^{\infty}}$ be the locally quaternion group and $D_{2^{\infty}}$ be the locally dihedral group then {\rm Pow(}$Q_{2^{\infty}})$ is isomorphic to {\rm Com(}$D_{2^{\infty}})$.
\end{thm}
\begin{proof}
    Let $Q_{2^{\infty}}$ be the locally quaternion group. Using the presentation of $Q_{2^{\infty}}$ given above, we get that $Q_{2^{\infty}} = C \langle x \rangle$ where $C$ is the Pr\"ufer 2-group. Since $|x|=4$ and $x^2 \in C$ is the unique involution in $Q_{2^{\infty}}$, we see that $Q_{2^{\infty}}$ is the disjoint union of cosets $C$ and $Cx$. Let $D_{2^{\infty}}$ be the locally dihedral group. By the presentation of $D_{2^{\infty}}$, we have that $D_{2^{\infty}}$ is the semi-direct product of $C$ and $\langle s \rangle$. This implies that $D_{2^{\infty}}$ is the disjoint union of cosets $C$ and $Cs$.
    
    
    Now, we define a map $f : Q_{2^{\infty}} \rightarrow D_{2^{\infty}}$ as $f(c) = c$ for every $c \in C$ and $f(cx) = cs$ otherwise. 
    This map is one-one and onto. Suppose that the map does not preserve edges between Pow($Q_{2^{\infty}}$) and Com($D_{2^{\infty}}$). Then, there exist elements $a,b \in Q_{2^{\infty}}$ and $f(a),f(b) \in D_{2^{\infty}}$ such that one pair is adjacent and the other is not. Let $a$ be adjacent to $b$ and $f(a)$ not adjacent to $f(b)$. This implies that for some suitable $n \in \mathbb{N}$, $a$ is adjacent to $b$ in Pow($Q_{2^{n+1}}$). By Theorem 1 in \cite{Paper_1}, restriction of $f$ to $Q_{2^{n+1}}$ is an isomorphism and $f(a),f(b) \in D_{2.2^n}$. This gives us that $f(a)$ is adjacent to $f(b)$ in Com($D_{2.2^n}$). But Com($D_{2.2^n}$) is a subgraph of Com($D_{2^{\infty}}$) which gives us a contradiction. Hence, if $a$ is adjacent to $b$ then $f(a)$ is adjacent to $f(b)$. The converse is similar. 
\end{proof}

Next, we consider the infinite quaternion group $Q_{\infty}$ and some of its properties. Consider the subgroup of ${\rm GL}(2,\mathbb{C})$ given by $G = \left\{ x_a = \begin{pmatrix}
  a & 0\\ 
  0 & a^{-1}
\end{pmatrix} \mid a \in \mathbb{C}^* \right\}$ and the matrix $j = \begin{pmatrix}
  0 & i\\ 
  i & 0
\end{pmatrix} \in {\rm GL}(2,\mathbb{C})$. Then, the subgroup $G$ is abelian and the matrices $j$ and $x_a$ satisfy the relations $j^2 = -I$ and $jx_aj^{-1}=x_a^{-1}=x_{a^{-1}}$ for every $x_a \in G$. We take $Q_{\infty} = G\langle j \rangle \leq {\rm GL}(2,\mathbb{C})$.

\begin{lemma}\label{q_infinity_prop}
    Let $Q_{\infty}$ be the infinite quaternion group described above. Then
    \begin{enumerate}[{\rm (i)}]
        \item Every element of $Q_{\infty}$ can be written as $x_aj^{\alpha}$ where $a \in \mathbb{C}^*$ and $\alpha = 0,1$.
        \item The element $x_a$ has finite order n if and only if $a$ is a primitive $n^{th}$ root of unity.
        \item The element $x_aj$ has order 4 and $\langle x_aj \rangle = \{ I, -I, x_aj, x_{a^{-1}}j \}$.
        \item The group $Q_{\infty}$ contains a copy of the dicyclic group $Q_{4m}$ for every $m$.
        \item The locally quaternion group $Q_{2^{\infty}}$ can be regarded as a subgroup of $Q_{\infty}$.
    \end{enumerate}
\end{lemma}

\begin{proof}
    The first part follows as $Q_{\infty} = G\langle j \rangle$ and $x_aj^2 = x_{-a} \in G$. For the second part, observe that we have $(x_a)^n = x_{a^n}$. The next part can be verified by direct calculation. Lastly, note that if $a$ is a primitive $2m^{th}$ root of unity, then the subgroup $\langle x_a,j \rangle$ is isomorphic to the dicyclic group $Q_{4m}$. Thus, every generalised quaternion group can be regarded as a subgroup of $Q_{\infty}$. Consequently, the group $Q_{\infty}$ will also have a copy of $Q_{2^{\infty}}$ as its subgroup.
\end{proof}

\begin{prop}
    Let $Q_{\infty}$ denote the infinite quaternion group and $Q_{4n}$ denote the dicyclic group.
    \begin{enumerate}[{\rm (i)}]
        \item $E$(Pow($Q_{\infty}$) $\subsetneq$ $E$(EPow($Q_{\infty}$) $\subsetneq$ $E$(Com$Q_{\infty}$).
        \item The structure of commuting graph of $Q_{4n}$ is carried forth to the commuting graph of $Q_{\infty}$.
    \end{enumerate}
\end{prop}

\begin{proof}
    The infinite quaternion group has both finite and infinite order elements and any two infinite order elements in it commute with each other. So, it has subgroups isomorphic to $\mathbb{Z}$ and $\mathbb{Z} \times \mathbb{Z}$. The first part follows by Theorems \ref{theo1} and \ref{theo2}. The identity and unique involution of $Q_{\infty}$ are universal vertices of Com($Q_{\infty}$). Since $G$ is abelian, the elements of $G$ form a complete subgraph on infinitely many vertices. The only other vertices are of the form $x_aj$ and it is easy to check that the only elements of $Q_{\infty}$ that commute with $x_aj$ are in $\langle x_aj \rangle$. Therefore, we have  Com($Q_{\infty}$) = $(K_{\infty}$ $\cup$ $\infty K_2)$ $\nabla$ $K_2$. We showed in Theorem 2 in \cite{Paper_1} that \newline Com($Q_{4n}$) = $(K_{2n-2}$ $\cup$ $n K_2)$ $\nabla$ $K_2$ which proves the rest. 
\end{proof}

The graph below depicts commuting graphs of $Q_{\infty}$ and $Q_{4n}$. For $Q_{4n}$, we have the complete graph on $2n$ vertices on the left and $n$ copies of $K_4$ on the right, whereas, for $Q_{\infty}$, we have a complete graph on uncountably many vertices and uncountably many copies of $K_4$. 

\begin{figure}[H]
    \centering
    \includegraphics[scale=0.35]{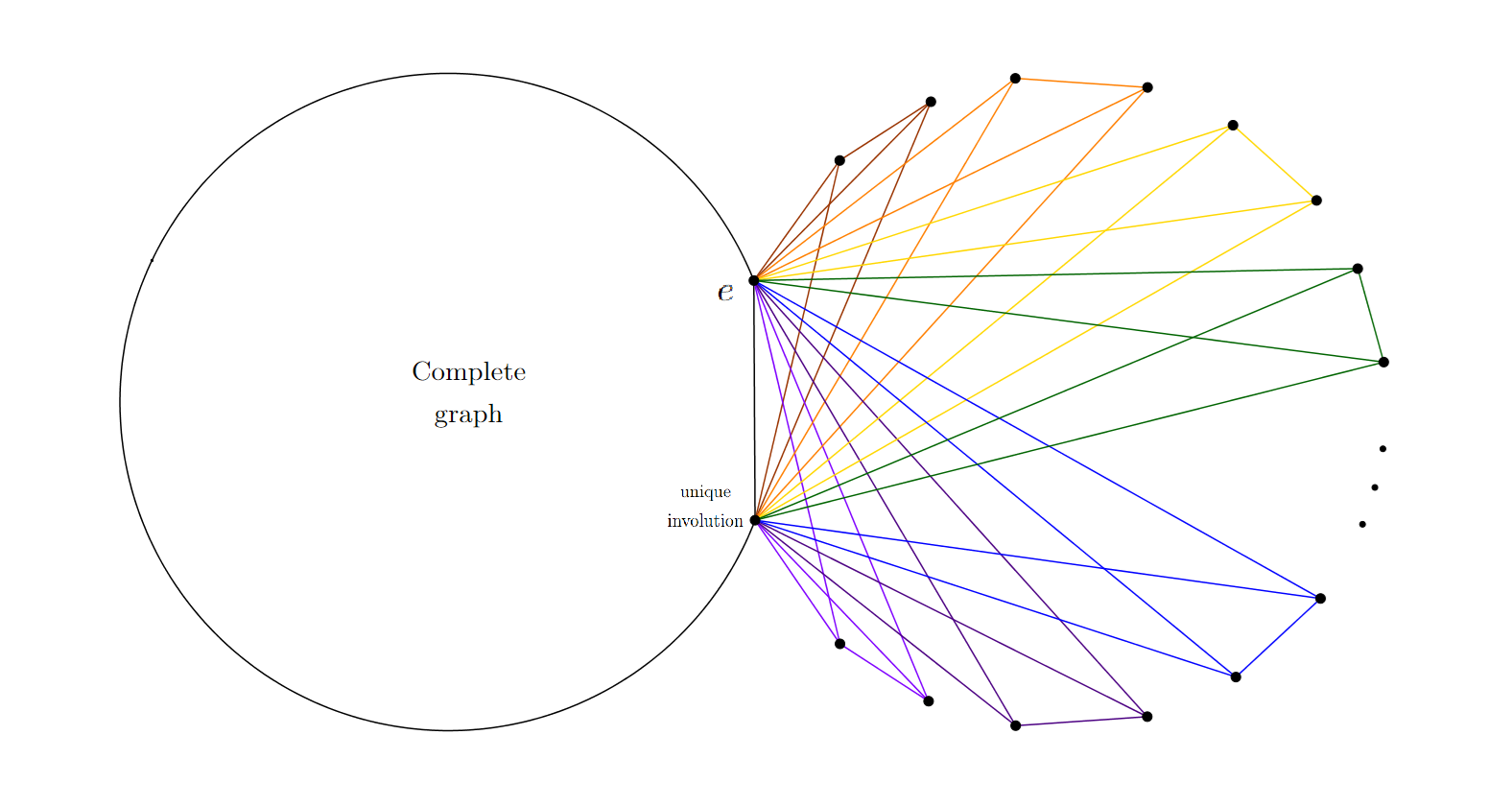}
    \caption{A depiction of the commuting graph of $Q_{\infty}$ and the power graph of $Q_{2^{\infty}}$}\label{com_g}
\end{figure}

In \cite{Paper_1}, we answered the following question posed by P. J. Cameron \cite{PJC_reviewpaper}: Do there exist groups $G$ and $H$ with $G \not\cong H$ such that the power graph of $G$ is isomorphic to the commuting graph of $H$? We proved that Pow($Q_{2^{n+1}}$) $\cong$ Com($D_{2.2^n}$) even though the two groups are not isomorphic to each other. For the generalised quaternion group, we have Pow($Q_{2^{n+1}}$) = Com($Q_{2^{n+1}}$). Observe that Com($Q_{2^{n+1}}$) $\cong$ Com($D_{2.2^n}$). In Theorem \ref{theo4} above, we showed that Pow($Q_{2^{\infty}}$) $\cong$ Com($D_{2^{\infty}}$). By Corollary \ref{Cor_1}, we have Pow($Q_{2^{\infty}}$) = Com($Q_{2^{\infty}}$). Observe that Com($Q_{2^{\infty}}$) is isomorphic to Com($D_{2^{\infty}}$).
Now, we give examples of groups $G$ and $H$ not isomorphic to each other such that Pow($G$) $\neq$ Com($G$) as well as Pow($H$) $\neq$ Com($H$) but Pow($G$) $\cong$ Com($H$).

The infinite dihedral group is defined as the group of symmetries of the frieze pattern $\ldots \bot\bot\bot \ldots$. This infinite group consists of the translations and reflections about the vertical axis. It has a presentation $$ D_{\infty} = \langle r,t \mid t^2 = 1 , trt = r^{-1} \rangle.$$ 
The translations have infinite order, and just like the finite dihedral groups, the reflections have order 2.

\begin{cor}
    Let $D_{\infty}$ be the infinite dihedral group. Then {\rm Pow(}$D_{\infty}${\rm )} $\neq$ {\rm EPow(}$D_{\infty}${\rm )} = {\rm Com(}$D_{\infty}${\rm )}
\end{cor}

\begin{proof}
    The infinite dihedral group is isomorphic to the semi-direct product of $\mathbb{Z}$ and $\mathbb{Z}_2$. As an immediate consequence of Theorem \ref{theo1}, we have Pow($D_{\infty}$) $\neq$ EPow($D_{\infty}$). Since $D_{\infty}$ is the disjoint union of cosets $\langle r \rangle$ and $\langle r \rangle t$, the elements are either translations denoted by $r^{\alpha}$, where $\alpha \in \mathbb{Z}$, or reflections given by $r^{\beta}t$, where $\beta \in \mathbb{Z}$. No translation commutes with a reflection and so, there is no subgroup isomorphic to $\mathbb{Z} \times \mathbb{Z}_p$ in $D_{\infty}$ for a prime $p$. Similarly, since no two reflections commute, we have no subgroup isomorphic to $\mathbb{Z}_p \times \mathbb{Z}_p$ and neither does $D_{\infty}$ have a subgroup isomorphic to $\mathbb{Z} \times \mathbb{Z}$. By Theorem \ref{theo2}, we have EPow($D_{\infty}$) = Com($D_{\infty}$).    
\end{proof}

It can be observed that only the identity is a universal vertex in Com($D_{\infty}$). In Com($D_{\infty}$), we have a complete subgraph on infinite vertices and infinite copies of $K_2$, all intersecting in identity. In Corollary \ref{Cor_1}, we showed that Pow($D_{2^{\infty}}$) = EPow($D_{2^{\infty}}$) $\neq$ Com($D_{2^{\infty}}$). In the power graph of $D_{2^{\infty}}$, we have a complete subgraph on infinite vertices representing the rotations and infinite copies of $K_2$ on the identity and each reflection. All of these graphs intersect in identity element of $D_{2^{\infty}}$. This gives us the following result, which answers the question posed by P. J. Cameron in \cite{PJC_reviewpaper}, namely, Can Pow($G_1$) and Com($G_2$) be isomorphic for non-isomorphic groups $G_1$ and $G_2$ which both have their respective power graphs not equal to their respective commuting graphs? 

\begin{thm} 
    Let $D_{\infty}$ be the infinite dihedral group and $D_{2^{\infty}}$ be the locally dihedral group. Then {\rm Pow(}$D_{2^{\infty}}$$)$ is isomorphic to {\rm Com(}$D_{\infty})$.
\end{thm}

\begin{proof}
    Let $D_{\infty}$ be the infinite dihedral group and $D_{2^{\infty}}$ be the locally dihedral group. To show that the power graph of $D_{2^{\infty}}$ is isomorphic to the commuting graph of $D_{\infty}$, we first show the existence of a bijective function between the two groups and later, we will show that the bijective map preserves edges. Consider the subgroup $\langle r \rangle$ of $D_{\infty}$ and the subgroup $C$ of $D_{2^{\infty}}$ (see presentations above). Since both of these subgroups are infinite countable, there exists a bijective function $f: \langle r \rangle \rightarrow C$. Using the facts established above that $D_{\infty}$ is the disjoint union of cosets $\langle r \rangle$ and $\langle r \rangle t$ and $D_{2^{\infty}}$ is the disjoint union of cosets $C$ and $Cs$, we can extend this bijective function from $D_{\infty}$ to $D_{2^{\infty}}$ by defining $f(r^it)=f(r^i)s$. It is easy to verify that $f$ is a bijective map from $D_{\infty}$ to $D_{2^{\infty}}$.\\
    Next, we show that this map is edge-preserving. Let $x,y$ be two distinct elements of $D_{\infty}$. If $x,y$ belong to $ \langle r \rangle$, then they commute with each other, and so they are adjacent in Com($D_{\infty}$). Their images $f(x)$ and $f(y)$ under the map $f$ belong to the subgroup $C$ of $D_{2^{\infty}}$, which implies that there exists some $m \in \mathbb{N}$ such that $f(x), f(y) \in \langle g_m \rangle$, where $g_m$ is one of the generators of the Pr\"ufer group $C$. Since $\langle g_m \rangle$ is a cyclic group of prime-power order and the power graph of a cyclic group of prime-power order is complete, therefore $f(x)$ is adjacent to $f(y)$ in Pow($D_{2^{\infty}}$). Now, let $x=r^it$ and $y=r^jt$ for some $i,j$. Since no reflection of $D_{\infty}$ commutes with any other non-identity element, we don't have an edge joining $x$ and $y$. Their images, $f(r^i)s$ and $f(r^j)s$, which are of order 2, are also not adjacent to each other in the power graph. Lastly, let $x=r^it$ and $y=r^j$ for some $i,j$. Then, they are not adjacent in the commuting graph of $D_{\infty}$ just as their images $f(r^i)s$ and $f(r^j)$ are not adjacent in the power graph of $D_{2^{\infty}}$. It is also clear that $D_{\infty}$ is not isomorphic to $D_{2^{\infty}}$ as the latter is a torsion group whereas $D_{\infty}$ is not. This completes the proof.  
\end{proof}

The graph below describes the commuting graph of $D_{\infty}$ and the power graph of $D_{2^{\infty}}$. For $D_{\infty}$, the complete graph on the left is on all the vertices corresponding to the translations and on the right, we have infinite copies of $K_2$ on all the reflections. For $D_{2^{\infty}}$, on the left we have a complete graph on all the rotations and infinite copies of $K_2$ on all the reflections on the right. 

\begin{figure}[H]
    \centering
    \includegraphics[scale=0.38]{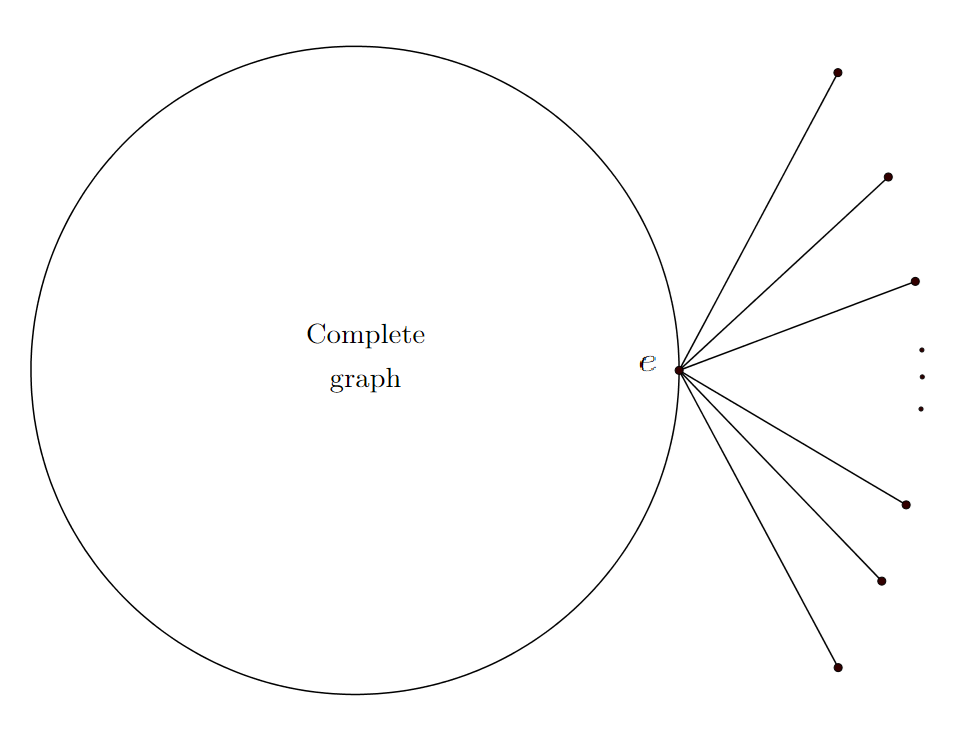}
    \caption{Commuting graph of $D_{\infty}$ and power graph of $D_{2^{\infty}}$}\label{Pow_g}
\end{figure}


\end{document}